\documentclass{amsart} 
\pdfoutput=1
\usepackage{graphicx} 
\usepackage{amsmath,amssymb,amsthm}
\usepackage[dvipsnames]{xcolor}
\usepackage{hyperref}
\usepackage{comment}

\title{An Infinite Family of Artin-Schreier Curves with Minimal a-number}
\author{Iris Shi}
\theoremstyle{plain}
\newtheorem{theorem}{Theorem}[section]
\newtheorem{lemma}[theorem]{Lemma}
\newtheorem{proposition}[theorem]{Proposition}

\theoremstyle{definition}
\newtheorem{example}[theorem]{Example}
\newtheorem{remark}[theorem]{Remark}
\newtheorem{notation}[theorem]{Notation}
\newtheorem{definition}[theorem]{Definition}

\newcommand{\Zp}[0]{\mathbb Z/p \mathbb Z}

\newcommand{\spec}[1]{\text{Spec}(#1)}

\DeclareFontFamily{OT1}{rsfs}{}
\DeclareFontShape{OT1}{rsfs}{n}{it}{<-> rsfs10}{}
\DeclareMathAlphabet{\mathscr}{OT1}{rsfs}{n}{it}

\usepackage{listings}
\usepackage{color}

\definecolor{dkgreen}{rgb}{0,0.6,0}
\definecolor{gray}{rgb}{0.5,0.5,0.5}
\definecolor{mauve}{rgb}{0.58,0,0.82}

\begin{document}

\begin{abstract}
    Let $p$ be an odd prime and $k$ be an algebraically closed field with characteristic $p$. Booher and Cais showed that the $a$-number of a $\mathbb Z/p \mathbb Z$-Galois cover of curves $\phi: Y \to X$ must be greater than a lower bound determined by the ramification of $\phi$. In this paper, we provide evidence that the lower bound is optimal by finding examples of Artin-Schreier curves that have $a$-number equal to its lower bound for all $p$. Furthermore we use formal patching to generate infinite families of Artin-Schreier curves with $a$-number equal to the lower bound in any characteristic.
\end{abstract}

\maketitle

\section{Introduction}
Let $p$ be an odd prime and $k$ be an algebraically closed field with characteristic $p$. Let $\phi: Y \to X$ be a smooth, projective, and connected cover of curves over $k$ with Galois group $G$. 
Some broad questions are
\begin{itemize}
    \item ``What properties of the curve $Y$ can be determined solely from properties of the curve $X$ and the map $\phi$?"
    \item ``What information is needed to determine the other properties of $Y$?" 
\end{itemize}
A classic version of this question concerns the genus of the curves, a standard numerical invariant associated to a curve. The genera of $X$ and $Y$ can be described as the $k$-dimension of $H^0(X,\Omega_X^1)$ and $H^0(Y,\Omega_Y^1)$, the space of regular $1$-forms on $X$ and $Y$, respectively. The well known Riemann-Hurwitz formula explains how the genus of $Y$ can be determined entirely from $X$ and ramification information about the cover $\phi$, 
\begin{equation*}
    2g_Y-2=|G|(2g_X-2)+\sum_{y \in \phi^{-1}(S)}\sum_{i \geq 0}(|G_i(y)|-1),
\end{equation*}
where $S$ is the branch locus of $\phi$ and $G_i(y)$ is the $i$th ramification group in lower numbering at $y$.

When $k$ has characteristic $p$, as in this paper, there are additional invariants arising from the Frobenius automorphism. We will work with the Cartier operator (which is dual to the Frobenius on $H^1(X,\mathcal O_X)$ via Serre duality). 
For the curve $X$, the Cartier operator is a $p^{-1}$-semilinear map $\mathcal C_X: H^0(X,\Omega_X^1) \to H^0(X,\Omega_X^1)$.
As $H^0(X,\Omega_X^1)$ is a finitely generated $k[\mathcal C_X]$-module, the structure theorem for finitely generated modules over a P.I.D. gives the following decomposition of $k[\mathcal C_X]$-modules,
\begin{equation}\label{kCX decomp}
    H^0(X,\Omega_X^1) = \bigoplus_i k[\mathcal C_X]/\mathcal C_X^{n_i} \oplus \bigoplus_j k[\mathcal C_X]/f_j(\mathcal C_X)^{n_j},
\end{equation}
where $f_j(\mathcal C_X)$ are irreducible polynomials in $k[\mathcal C_X]$ not equal to $\mathcal C_X$. (Although $k[\mathcal C_X]$ is not technically a P.I.D. as its non-commutative, there exists an identical structure theorem for non-commutative P.I.D. \cite[Theorem 3.19]{theoryofrings}.) Note that $\mathcal C_X$ acts nilpotently on the first part of \eqref{kCX decomp}. The $p$-rank of $X$, which we denote by $s_X$, is then defined as the $k$-dimension of $\bigoplus_j k[\mathcal C_X]/f_j(\mathcal C_X)^{n_j}$, the second half of the decomposition in \eqref{kCX decomp}. Like the genus, the $p$-rank is another invariant of $Y$ that can often be determined from $X$ and ramification information from $\pi$. The Deuring-Shafarevich formula says that when $G=\Zp$,
\begin{equation*}
    s_Y-1=p(s_X-1)+\sum_{y\in \phi^{-1}(S)}(d_y-1),
\end{equation*}
where $S$ is the branch locus of $\phi$ and $d_y$ is the unique break in the ramification filtration at $y$ \cite{subrao}. We refer to $d_y$ as the \textit{ramification break} at $y$.

The $a$-number is an additional numerical invariant describing the structure of $H^0(X,\Omega_X^1)$, defined as the number of summands in $\bigoplus_i k[\mathcal C_X]/\mathcal C_X^{n_i}$, the first half of \eqref{kCX decomp}, i.e. the $k$-dimension of $\ker(\mathcal C_X)$. This invariant is less understood and the focus of this paper. Since the $a$-number is similar to the $p$-rank in definition and simplicity, it would be natural to attempt to find an analog of the Deuring-Sharfarevich formula. 
However, no such formula exists, and we can use Artin-Schreier curves to see this. Artin-Schreier curves are smooth, projective, connected covers of $\mathbb P^1$ with Galois group $\Zp$ and are a key case of $p$-group curves. Any Artin-Schreier curve can be defined by an equation of the form $y^p-y=f$ where $f\in k(x)$ is nonconstant and $k(x)$ is the function field corresponding to $\mathbb P^1$. The ramifcation break at $y$, $d_y$, is the order of the pole of $f$ at a branch point $y$. The Artin-Schreier curve defined by $y^3-y=x^7$ has $a$-number $4$ while the Artin-Schreier curve defined by $y^3-y=x^7+x^5$ has $a$-number $3$, despite both being covers of $\mathbb P^1$ and being branched only over $\infty$ with ramification break $d_\infty=7$. This shows that the $a$-number of $Y$ cannot be determined using the same information needed to determine its genus and $p$-rank. 

However $X$ and the ramification of $\phi$ still constrain the $a$-number of $Y$.  Farnell and Pries \cite{farnellpries} discovered a formula for the $a$-number of an Artin-Schreier curve dependent only on $X$ and its ramification information for a specific congruence condition on the order of the poles of its defining equation. 
Elkin and Pries \cite{elkinpries} found a specific formula for the $a$-number of hyperelliptic $k$-curves in characteristic 2 dependent only on its ramification information. These are specific cases where the $a$-number of $Y$ can be determined given some condition on the characteristic or the ramification. 
In general, Booher and Cais \cite{boohercais} were able to find upper and lower bounds for the $a$-number of a curve $Y$, where $\phi: Y\to X$ is a branched $\Zp$-cover, depending only on $X$ and the ramification of $\phi$,
\begin{equation} \label{initiallb}
    \max_{1\leq j \leq p-1}\left(
    \sum_{Q\in S}\sum_{i=j}^{p-1}\left(\left\lfloor\frac{id_Q}{p}\right\rfloor -\left\lfloor\frac{id_Q}{p}-\left(1-\frac{1}{p}\right)\frac{jd_Q}{p}\right\rfloor\right)\right)
    \leq a_Y,
\end{equation}
\begin{equation}\label{initiallb2}
    a_Y 
    \leq pa_X + \sum_{Q\in S}\sum_{i=1}^{p-1}\left(\left\lfloor\frac{id_Q}{p}\right\rfloor - (p-i)\left\lfloor\frac{id_Q}{p^2}\right\rfloor\right),
\end{equation}
where $S$ is the branch locus and $d_Q$ is the ramification break at $Q$ in lower numbering. If the branch locus contains only a single point $Q$, we refer to the ramification break at $Q$ as $d$ and to the lower bound in the left side of \eqref{initiallb} as $L(d)$.

In some sense, the bounds in \eqref{initiallb} and \eqref{initiallb2} are an analog to the Deuring-Shaferavich for the $a$-number. However, because they are not an exact formula, work needs to be done to show that they are the optimal bounds. 
Experimentally, the bounds seem to be optimal. A randomly generated curve over $\mathbb F_p$ with a fixed ramification break, $d$, has a high chance of having an $a$-number equal to the lower bound. The code to check this can be found at \cite{github} or attached to the arXiv submission. Some experiments are also tabled in \cite[Remark 1.5]{promys}.
The authors of \cite{promys} also provide harder evidence that the lower bound is optimal for small $p$, as they exhibit curves over $\mathbb F_3$ and $\mathbb F_5$ with any size ramification break that attains an $a$-number equal to the lower bound in \eqref{initiallb}. 
In this paper, we provide additional cases in which the bound in \eqref{initiallb} is optimal.  

\newtheorem*{theorem:family}{Theorem \ref{family}}
\begin{theorem:family}
    For any prime $p$ and any positive $d \equiv -1 \pmod {p^2}$ there exists a $\mathbb Z/p\mathbb Z$-Galois cover $X \to \mathbb P^1$ branched at one point with ramification break $d$ and $a$-number equal to $L(d)$.
\end{theorem:family}

\newtheorem*{theorem:family2}{Theorem \ref{farnell family}}
\begin{theorem:family2}
    For any prime $p$ and any positive $d \equiv p - 1 \pmod {p^2}$ there exists a $\Zp$-Galois cover $X \to \mathbb P^1$ branched at one point with ramification break $d$ and $a$-number equal to $L(d)$.
\end{theorem:family2}

\newtheorem*{theorem:family3}{Theorem \ref{code examples}}
\begin{theorem:family3}
    For any prime $p \leq 23$ and any positive $d \equiv -1 \pmod p$, there exists a $\Zp$-Galois cover $X \to \mathbb P^1$ branched at one point with ramification break $d$ and $a$-number equal to $L(d)$.
\end{theorem:family3}

In Theorem \ref{family} and \ref{farnell family}, we generate families of curves with arbitrarily large ramification break in any finite characteristic that have $a$-number equal to the lower bound in \eqref{initiallb}. We do this by exhibiting small curves with minimal $a$-number and then using formal patching to build an infinite family of curves by combining the smaller curves together inductively. Note that our method generates curves of arbitrarily large ramification breaks for any prime $p$, as opposed to just $p=3$ and $p=5$. This serves as evidence that the bounds are optimal. 

\begin{remark}
    The method used in Section \ref{infinite family} can only produce curves with ramification break $d \equiv -1 \pmod{p}$. For generic $p$, we only tackle $d \equiv \pm 1 \pmod{p^2}$ in this paper due to the difficulty in finding curves that have $a$-number equal to the lower bound for all $p$ that are easy to analyze. We produce curves for all of the (mod $p^2$) congruence classes for small $p$ in Theorem \ref{code examples} by using the MAGMA computational algebra system \cite{MAGMA} to find examples.
\end{remark}

\begin{remark}
    In Section \ref{infinite family}, we form covers with $a$-number equal to the lower bound with a single branch point by combining curves with the same branch point. However, we could also construct covers with multiple branch points having minimal $a$-number. To do so, combine the examples with a single branch point that we give in this paper using the formal patching arguments in \cite{booherpries,promys}. 
\end{remark}

\subsection*{Acknowledgements}
    The initial idea for this paper came from Rachel Pries, who I would like to thank for both the inspiration and her helpful review. I would also like to thank my advisor Jeremy Booher for his support and indispensable guidance.

\section{Small Artin-Schreier Curves}

\subsection{Numerical invariants of Artin-Schreier curves}
Fix an odd prime $p$. Let $k$ be an algebraically closed field with characteristic $p$. An Artin-Schreier curve is a finite morphism of smooth, projective, connected curves $X \to \mathbb P^1$ with Galois group $\Zp$. Any Artin-Schreier curve can be defined by an equation of the form $y^p-y = f(x)$ where $f(x) \in k(x)$ is nonconstant, and $k(x)$ is the function field of $\mathbb P^1$. There are several numerical invariants associated to an Artin-Schreier curve, namely the genus, $p$-rank, and $a$-number. These invariants can be viewed as properties of the Cartier operator and the space of regular $1$-forms, $H^0(X,\Omega_X^1)$.\\

The Cartier operator, $\mathcal C_X: H^0(X,\Omega_X^1) \to H^0(X,\Omega_X^1)$, is a semi-linear operator  with the following properties. For rational functions $f, x$ and differentials $\alpha,\beta$:
\begin{align}\label{cartier properties}
\mathcal C_X(f^p\alpha+\beta) &=f \mathcal C_X(\alpha) +\mathcal C_X(\beta),\nonumber \\
\mathcal C_X(x^{p-1}dx) &= dx,\\
\mathcal C_X(x^ndx) &= 0 \text{ if } n \not\equiv -1 \pmod p.\nonumber
\end{align}

The genus, $g$, of $X$ is equal to the dimension of $H^0(X,\Omega_X^1)$. The $p$-rank is equal to the dimension of the image of $\mathcal C_X^g$. The $a$-number, written as $a(X)$, is equal to the dimension of the kernel of the Cartier operator.

The $a$-number for an Artin-Schreier curve with branch locus $S$ and $D=\{d_Q\}_{Q\in S}$, the set of ramification breaks over $S$, has the lower bound $L(D)$ \cite{boohercais},
\begin{equation*}
    L(D) = \max_{1\leq j \leq p-1}\left(
    \sum_{Q\in S}\sum_{i=j}^{p-1}\left(\left\lfloor\frac{id_Q}{p}\right\rfloor -\left\lfloor\frac{id_Q}{p}-\left(1-\frac{1}{p}\right)\frac{jd_Q}{p}\right\rfloor\right)\right).
\end{equation*}
We will only examine curves branched at one point. So in this case $D=\{d\}$ is a singleton set and we denote the lower bound as $L(d)$. We will use the following simplified formula for $L(d)$ in calculations.

\begin{lemma}\label{simplified formula}
Let $\pi: Y \to X$ be a finite morphism of smooth, projective, and geometrically connected curves over a perfect field with odd characteristic $p$ and Galois with group $\Zp$ branched at a single point. Let $d \in \mathbb N$ be the ramification break over that point. Then,
\begin{equation*}
    L(d)=\sum_{i=\frac{p+1}{2}}^{p-1}\left(\left\lfloor\frac{id}{p}\right\rfloor -\left\lfloor\frac{id}{p}-\left(1-\frac{1}{p}\right)\frac{(p+1)d}{2p}\right\rfloor\right).
\end{equation*}
\end{lemma}

\begin{proof}
    \cite[Cor 2.15]{promys}.
\end{proof}

\subsection{Computations with Artin-Schreier curves}

Let $p$ be an odd prime and let $k$ be an algebraically closed field with characteristic $p$. For an Artin-Schreier curve $X$ defined by $y^p-y=f$ such that $f \in k[x]$ has a pole of order $d$ at infinity, the set $\mathcal B_X$, defined by 
\begin{equation}\label{basis}
    \mathcal B_X:=\left\{y^ix^jdx: 0 \leq i \leq p-2, 0 \leq j \leq \left\lceil \frac {(p-i-1)d}{p} \right\rceil -2 \right\},
\end{equation}
is a basis for $H^0(X,\Omega_X^1)$ \cite[Lemma 3.7]{boohercais}. We will examine the specific cases $d=p^2-1$ and $d=p^2+1$, so the following lemmas will be computationally helpful. 

\begin{lemma}\label{basis size}
For odd prime $p$ and integers $0 \leq i \leq p-2$, the following two equations hold:
\begin{equation*}
    \left\lceil \frac {(p-i-1)(p^2-1)}{p} \right\rceil-2=p^2-(1+i)p-2,
\end{equation*}
\begin{equation*}
    \left\lceil \frac {(p-i-1)(p^2+1)}{p} \right\rceil-2=p^2-(1+i)p-1.
\end{equation*}
\end{lemma}

\begin{proof}
    Elementary.
\end{proof}

\begin{lemma}\label{simplified lower bounds}
If $p$ is an odd prime, then
\begin{equation*}
L(p^2 + 1) = L(p^2 - 1) = \left(\frac{p-1}{2}\right) \frac{p^2-1}{2}.
\end{equation*}
\end{lemma}

\begin{proof}
    We only prove $L(p^2+1)= \left(\frac{p-1}{2}\right) \frac{p^2-1}{2}$. The proof showing $L(p^2-1)= \left(\frac{p-1}{2}\right) \frac{p^2-1}{2}$ is similar. 
    Setting $d=p^2+1$ in Lemma \ref{simplified formula} and simplifying gives
    \begin{equation*}
        L(p^2 + 1)= \sum_{i=\frac{p+1}{2}}^{p-1} ip +\left\lfloor \frac{i}{p}\right\rfloor -\left(ip + \frac{-p^2+1}{2} + 
        \left\lfloor- \frac{-2ip+p^2-1}{2p^2}\right\rfloor\right).
    \end{equation*}
    Since $\frac{p+1}{2} \leq i \leq p-1$, the range for the numerator of the floor term will be $0 < -(-2ip+p^2-1) < 2p^2$. Hence $\left\lfloor \frac{i}{p}\right\rfloor = \left\lfloor -\frac{-2ip+p^2-1}{2p^2}\right\rfloor = 0$. Simplifying gives 
    \begin{equation*}
    L(p^2-1) = \sum_{i=\frac{p+1}{2}}^{p-1} \frac{p^2-1}{2} = \left(\frac{p-1}{2}\right) \frac{p^2-1}{2}. \qedhere
    \end{equation*}
\end{proof}

\begin{lemma}\label{farnellcase}
If $p$ is an odd prime, then

\begin{equation*}
    L(p - 1) = \frac{(p-1)^2}{4}.
\end{equation*}
\end{lemma}

\begin{proof}
    The proof of Lemma \ref{farnellcase} is similar to the proof of Lemma \ref{simplified lower bounds}.
\end{proof}

\begin{lemma} \label{addingpsquared}
If $p$ is an odd prime and $d$ is a positive integer,
\begin{equation*}
     L(d+p^2)=L(d)+L(p^2+1)   .
\end{equation*}
\end{lemma}

\begin{proof}
    Plugging $d+p^2$ into Lemma \ref{simplified formula} gives
        \begin{align*}
        L(d+p^2)&=\sum_{i=\frac{p+1}{2}}^{p-1}\left\lfloor\frac{i(d+p^2)}{p}\right\rfloor -\left\lfloor\frac{i(d+p^2)}{p}-\left(1-\frac{1}{p}\right)\frac{(p+1)(d+p^2)}{2p}\right\rfloor\\
        &=\sum_{i=\frac{p+1}{2}}^{p-1}\left\lfloor\frac{id}{p}\right\rfloor -\left\lfloor\frac{id}{p}-\left(1-\frac{1}{p}\right)\frac{(p+1)d}{2p}\right\rfloor+\frac{p^2-1}{2}\\
        &=L(d)+\left(\frac{p-1}{2}\right)\frac{p^2-1}{2}.
        \end{align*}
    Hence, with Lemma \ref{simplified lower bounds}, $L(d+p^2)=L(d)+L(p^2+1)$.
\end{proof}

We will use the lexicographic ordering on $\mathcal B_X$ with $y>x$. i.e. for differentials $y^ix^jdx$ and $y^ax^bdx$, if $y^ix^jdx > y^ax^bdx$ then either $i>a$ or $i=a$ and $j>b$. For any element $\alpha \in \mathcal B_X$, define $B_\alpha$ to be the set of basis vectors smaller than $\alpha$ under the lexicographic order and denote the span of the image of a set under the Cartier operator by $\text{Span}(\mathcal C_X(A))$. In other words, we write
\begin{equation*}
B_\alpha:= \{\beta \in \mathcal B_X: \beta < \alpha\},
\end{equation*}
\begin{equation*} 
\text{Span}(\mathcal C_X(A)): = \text{Span}(\{\mathcal C_X(a): a \in A\}).
\end{equation*}

The defining equation $y^p-y=f$ with $f \in k[x]$ can be used to rewrite differentials in $H^0(X,\Omega_X^1)$. For $y^mx^ndx \in H^0(X,\Omega_X^1)$, the differential can be rewritten $y^mx^ndx=(y^p-f)^mx^ndx$. This equivalence is useful in calculations with the Cartier operator since the rewritten form works well with the properties in \eqref{cartier properties} to determine what the Cartier of any differential is. The following theorems concern cases where $f$ is a binomial and thus the following fact will be useful.

\begin{lemma}\label{binomial formula}
Fix odd prime $p$, distinct positive integers $d,e\not\equiv 0 \pmod p$ and nonzero $a,b \in k$. Let $X \to \mathbb P^1$ be the Artin-Schreier cover defined by the equation $y^p-y=-(ax^d+bx^e)$. A differential $y^mx^ndx \in H^0(X,\Omega_X^1)$ can be expressed as 
    \begin{equation*}
    y^mx^ndx  = \sum_{i=0}^m\sum_{j=0}^{m-i} {m\choose i}{m-i\choose j} a^{m-i-j} b^j y^{pi} x^{(m-i-j)d+je+n}dx.
    \end{equation*}

\end{lemma}

\begin{proof}
    Since $y^m x^n dx \in H^0(X,\Omega_X^1)$, we can use the equation $y^p-y=ax^d+bx^e$ to substitute $y^m x^n dx= (y^p+(ax^d+bx^e))^m x^n dx$. The binomial theorem then gives the desired formula.
\end{proof}

\subsection{Exhibiting curves with minimal $a$-number}

In this section, we exhibit curves for all odd primes $p$ with some fixed ramification break $d$ such that their $a$-number is equal to the lower bound $L(d)$. These curves will be the base curves that combine to get an infinite family of curves in Section \ref{patching curves}.

\begin{proposition} \label{farnellresult}
    Let $p$ be an odd prime. If $f(x) \in k[x]$ has degree $p-1$, then the Artin-Schreier cover $X\to \mathbb P^1$ defined by $y^p-y=f(x)$ has a-number equal to $L(p-1)$.
\end{proposition}

\begin{proof}
    By Lemma \ref{farnellcase}, we get the lower bound 
    \begin{equation*}
    L(p - 1) = \frac{(p-1)^2}{4}.
    \end{equation*}
    This is the $a$-number of any Artin-Schreier cover defined by $y^p-y=f(x)$ such that $f(x) \in k[x]$ has degree $p-1$ \cite[Theorem 3.9]{farnellpries}.
\end{proof}

\begin{lemma}\label{plusmatrixpart1}
Let $p$ be an odd prime and $d=p^2+1$. Let $X \to \mathbb P^1$ be the Artin-Schreier cover defined by the equation $y^p-y=-x^{d}-x^{d/2+p}$. Let $0 \leq k < \frac{p-1}{2}$ and $0 \leq l \leq \frac{d(p/2-1-k)}{p}$.
Then there is $y^mx^ndx \in \mathcal B_X$ with $m < \frac{p-1}{2}$ such that the largest term of $\mathcal C_X(y^mx^ndx)$ is $y^kx^ldx$.
\end{lemma}

\begin{proof}
    Note that $\mathcal C_X(y^{pk}x^{lp+(p-1)}dx)=y^kx^ldx$ using the properties of the Cartier operator in \eqref{cartier properties}. It is necessary to find $y^mx^ndx$ such that its expansion has the term $y^{pk}x^{lp+(p-1)}dx$. Let $a=\left\lfloor\frac{2lp+p^2-1}{2(p^2+1)}\right\rfloor$. Note that $0\leq a \leq \frac{p-3}{2}$. Let 
    \begin{equation*}
        b = \begin{cases}
            0 & \text{if } a=0\\
            ad & \text{if } a\neq 0 \text{ and } ad\geq lp+(p-1)\\
            (a-1)d + \frac{d}{2}+p & \text{if } a\neq 0 \text{ and } ad < lp+(p-1)
        \end{cases}.
    \end{equation*}
    Set $m=a+k$. Set $n=lp+(p-1)-b$. Check using \eqref{basis} that $y^mx^ndx \in \mathcal B_X$ with $m <\frac{p-1}{2}$.
    Using Lemma \ref{binomial formula}, we write 
    \begin{equation} \label{exponents2}
        y^mx^ndx  = \sum_{i=0}^m\sum_{j=0}^{m-i} {m\choose i}{m-i\choose j} y^{pi}x^{(m-i)d-j(d/2-p)+n}dx.
    \end{equation}
    Choosing $i = k$ and $j=\frac{2ad-2b}{d-2p}$ gives the desired term of $y^{pk}x^{lp+(p-1)}dx$. Recall from \eqref{cartier properties} that $\mathcal C_X(y^{pe}x^fdx) \neq 0$ if and only if $f \equiv -1 \pmod p$. So, showing $y^kx^ldx$ is the largest term of $\mathcal C_X(y^mx^ndx)$ is equivalent to showing $i=k$ is the largest $0\leq i \leq m$ such that there is $0\leq j\leq m-i$ with $(m-i)d-j(\frac d 2-p)+n \equiv -1 \pmod p$ and $j=\frac{2ad-2b}{d-2p}$ is the smallest such $j$. Let $i=k+c$ be the largest such $i$. Note that $0\leq c \leq a$ and $0 \leq j \leq a-c$.
    Substituting $m,n,c$ into the congruence and simplifying gives $j\equiv 2a-2c-2b \pmod p$. If $b=0$, then $a=c=j=0$ and $i=k$ as desired. If $b=ad$ we get that $j\equiv -2c \pmod p$. Due to the bounds on $a,c,j$, we get $j=c=0$ and $i=k$ as desired. If $b=(a-1)d+\frac{d}2+p$, then we have $j\equiv 1-2c \pmod p$. Again, due to the bounds on $a,c,j$ we get $c=0$. Hence, $j=1$ and $i=k$ as desired. In all cases we find that the choice of $i=k$ and $j=\frac{2ad-2b}{d-2p}$ corresponds to the largest term of $\mathcal C_X(y^mx^ndx)$. Hence $y^mx^ndx$

\end{proof}

\begin{lemma}\label{plusmatrixpart2}
    Let $p$ be an odd prime and $d=p^2+1$. Let $X \to \mathbb P^1$ be the Artin-Schreier cover defined by the equation $y^p-y=-x^{d}-x^{d/2+p}$. For all $\alpha = y^mx^ndx \in \mathcal B_X$ with $m \geq \frac {p-1}{2}$, $\mathcal C_X(\alpha) \not\in \emph{\text{Span}}(\mathcal C_X(\mathcal B_\alpha))$.
\end{lemma}

\begin{proof}
    Fix $\alpha = y^mx^ndx$ with $m\geq \frac{p-1}{2}$. We first show that $\mathcal C_X(\alpha) \neq 0$. By Lemma \ref{binomial formula}, $\alpha$ can be expanded using the formula in \eqref{exponents2}. A summand in \eqref{exponents2} corresponds to a pair $(i,j)$ with $0\leq i \leq m$ and $0 \leq j \leq m-i$. The $y$ exponent of such a summand is $pi$ and thus varies only with $i$. The $x$ exponent of a summand is given by $(m-i)d -j(\frac{d}{2}-p)+n$. Note then, that for a fixed $i$, the $x$ exponent varies uniquely depending on $j$ since $\frac{d}{2}-p \neq 0$. Hence each summand has a unique combination of exponents and does not combine with any other summand. Let 
    \begin{equation*}
    I_m := \{(i,j) \in \mathbb Z_{\geq 0}: m-\frac{p-1}{2} \leq i \leq m, 0 \leq j \leq 1\} - \{(m,1)\}.
    \end{equation*}
    Every $(i,j) \in I_m$ gives $(m-i)d - j (\frac d 2 - p) +n$ a distinct class (mod $p$). Hence there is always a summand from \eqref{exponents2} with exponent of the $x$ term congruent to $ - 1 \pmod p$. Hence every basis element $y^mx^ndx$ with $m\geq \frac{p-1}{2}$ has $\mathcal C_X(y^mx^ndx) \neq 0$.
    
    We now show that there is a term in the expansion of $\mathcal C_X(\alpha)$ that does not appear in $\text{Span}(\mathcal C_X(\mathcal B_\alpha))$. Choose $0\leq i \leq m$ to be the largest integer such that there exists $0 \leq j \leq m-i$ with $(m-i)d - j (\frac d 2 - p) +n \equiv -1 \pmod p$. Choose $j$ to be as large as possible. This exists since $\mathcal C_X(\alpha)$ is nonzero. Let $\beta \in \mathcal B_\alpha$ with $\beta = y^kx^ldx$, so either $k < m$ or $k=m$ and $l <n$. 
    Assume by way of contradiction that the coefficient of $y^{pi}x^{(m-i)d - j (\frac d 2 - p) +n}dx$ in the expanded form of $\beta$ is nonzero. This implies that $0 \leq i \leq k$ and there exists $0\leq g \leq k-i$ and $(k-i)d-g (\frac d 2 -p) +l = (m-i)d - j (\frac d 2 - p) +n$.
    Solving for $j$ gives $j(\frac d 2 - p)  = d(m-k)+g(\frac d 2 - p)+n-l$.
    Assume $k=m$ and $l<n$, then using \eqref{basis} we get bounds
    \begin{align*}
        0\leq n,l & \leq p^2-\left(1+\frac {p-1}{2}\right)p-2 = \frac {p^2-p}{2}-2,
    \end{align*}
    which implies $n-l\leq \frac {p^2-1}{2}-2$. However since $j$ is an integer, $(\frac d 2-p)$ must divide $n-l$. This implies $n-l=0$, contradicting $l<n$. 
    Now assume $k<m$. Note that $j = (d(m-k)+g(\frac{d}{2}-p)+n-l)/(\frac d 2 - p) \geq (d(m-k)-l)/(\frac{d}{2}-p)$. If $k<\frac{p-1}{2}$, then $m-k\geq 2$. If $k\geq \frac{p-1}{2}$, then from \eqref{basis} and Lemma \ref{basis size}, we get $l \leq \frac{p^2-3}{2}$. In either case we have $j\geq (d(m-k)-l)/(\frac{d}{2}-p)>1$.
    Now observe that $(m-(i+1))d - (j-2) (\frac d 2 - p) +n\equiv -1 \pmod p$ and $0 \leq i+1\leq m$ and $0 \leq j-2 \leq m - (i+1)$. Hence the maximality of $i$ is violated, since $(i+1)$ also satisfies the necessary condition. By way of contradiction, the coefficient of $y^{pi}x^{(m-i)d - j (\frac d 2-p) +n}dx$ is $0$ in the expanded form for any differential in $\mathcal B_\alpha$. Hence for all $\alpha = y^mx^ndx \in \mathcal B_X$ with $m \geq \frac {p-1}{2}$, $\mathcal C_X(\alpha) \not\in \text{Span}(\mathcal C_X(\mathcal B_\alpha))$.
\end{proof}

\begin{proposition}\label{bc2}
    Let $p$ be an odd prime. If $d=p^2 + 1$, the Artin-Schreier cover 
    $X\to \mathbb P^1$ defined by $y^p-y=-x^{d}-x^{d/2+p}$ has a-number equal to $L(d)$.
\end{proposition}

\begin{proof}
    An upperbound for the $a$-number can be found by computing a lower bound for the dimension of the image of the Cartier operator by rank-nullity. Fix differential $\alpha = y^{\frac {p-1}{2}}dx$.
    By Lemma \ref{plusmatrixpart1}, there are $\sum_{k=0}^{(p-3)/2} \left\lfloor \frac{d(p/2-1-k)}{p}+1 \right\rfloor$ many $\omega \in \mathcal B_X$ such that $\omega < \alpha$ and $\mathcal C_X(\omega)$ have distinct largest terms. 
    Thus the dimension of $\text{Span}(\mathcal C_X(\mathcal B_{\alpha}))$ is bounded from below by $\sum_{k=0}^{(p-3)/2} \left\lfloor \frac{d(p/2-1-k)}{p}+1\right\rfloor$. 
    Using Lemma \ref{plusmatrixpart2}, we get a lower bound for the dimension of the image of the Cartier operator given by $\dim(\text{Span}(\mathcal C_X(\mathcal B_X))) \geq \dim(\text{Span}(\mathcal C_X(\mathcal B_\alpha))) +|\{\omega \in \mathcal B_X: \omega \geq y^{\frac {p-1}{2}}dx\}|$. We can then substitute in the lower bound for $\text{Span}(\mathcal C_X(\mathcal B_\alpha))$ and use Lemma \ref{basis size} to count the number of basis elements in $\{\omega \in \mathcal B_X: \omega \geq y^{\frac {p-1}{2}}dx\}$.
    Denote this lower bound as $R(X)$:
    \begin{equation}\label{carterdecompplus}
        R(X) = \sum_{k=0}^{(p-3)/2} \left\lfloor \frac{d(\frac{p-2}{2}-1-k)}{p}+1 \right\rfloor + \sum_{k=(p-1)/2}^{p-2} (p^2-(1+k)p). 
    \end{equation}
    An upper bound for the $a$-number can be found by subtracting this lower bound from the genus.
    Note that using \eqref{basis} and Lemma \ref{basis size}, the genus can be written as $g_X=\sum_{k=0}^{p-2}(p^2-(1+k)p)$. Hence we can write the upper bound for $a(X)$ as
    \begin{equation*}
        a(X) \leq \ \  \left(\sum_{k=0}^{p-2} \ \ (p^2-(1+k)p)\right) - R(X).
    \end{equation*}
    The second sum in \eqref{carterdecompplus} cancels with the sum from the genus to give a sum with the same indices as the first sum in \eqref{carterdecompplus}. This can be further simplified by recalling $d = p^2+1$ to get
    \begin{align*}
         a(X) & \leq \sum_{k=0}^{(p-3)/2} \left(p^2-(1+k)p - \left\lfloor \frac{d(p/2-1-k)}{p} +1\right\rfloor\right)\\
         & \leq \sum_{k=0}^{(p-3)/2} p^2-(1+k)p -
         \left\lfloor \frac{(p^2+1)(p/2-1-k)}{p}+1\right\rfloor\\
         & \leq \sum_{k=0}^{(p-3)/2} p^2-(1+k)p -
         \left\lfloor \frac{p^2+1}{2}-(1+k)p+\frac{-1-k}{p}+1\right\rfloor\\
         & \leq \sum_{k=0}^{(p-3)/2} \frac{p^2-1}{2} -
         \left\lfloor \frac{-1-k}{p}+1\right\rfloor
    \end{align*}
    Since $0 \leq k \leq \frac{p-3}{2}$, we get that $\left\lfloor \frac{-1-k}{p}+1\right\rfloor = 0$. Thus simplifying gives 
    \begin{equation*}
        a(X) \leq \sum_{k=0}^{(p-3)/2} \frac{p^2-1}{2} = \left(\frac{p-1}{2}\right) \frac{p^2-1}{2} = L(d).
    \end{equation*}
    Since $a(X) \geq L(d)$ \cite[Theorem 1.1]{boohercais}, we get $a(X)=L(d).$
\end{proof}

The case for $d=p^2-1$ is proven with a similar approach to the case for $d=p^2+1$.

\begin{lemma}\label{matrixpart1}
    Let $p$ be an odd prime and $d=p^2-1$. Let $X \to \mathbb P^1$ be the Artin-Schreier cover defined by the equation $y^p-y=-x^{d}-x^{d/2}$. Let $0 \leq k < \frac{p-1}{2}$ and $0 \leq l \leq \frac{d(p/2-1-k)-p}{p}$.
    Then there is $y^mx^ndx \in \mathcal B_X$ with $m < \frac{p-1}{2}$ such that the largest term of $\mathcal C_X(y^mx^ndx)$ is $y^kx^ldx$.
\end{lemma}
\begin{proof}
    The proof of Lemma \ref{matrixpart1} is similar to the proof of Lemma \ref{plusmatrixpart1} except we set $a=\left\lfloor\frac{2lp+2(p-1)}{d}\right\rfloor$, $m=\left\lfloor\frac{a}{2}\right\rfloor+k$, and $n=lp+(p-1)-a\frac{d}2$ to get the desired $y^mx^ndx$. In this case, expanding $y^mx^ndx$ using Lemma \ref{binomial formula} gives 
    \begin{equation} \label{exponentsminus}
    y^mx^ndx  = \sum_{i=0}^m\sum_{j=0}^{m-i} {m\choose i}{m-i\choose j} y^{pi}x^{(m-i)d-j\frac{d}{2}+n}dx
    \end{equation}
    and choosing $i=k$ and $j=2\lceil\frac{a}{2}\rceil-a$ will give the term $y^{pk}x^{lp+(p-1)}dx$ as needed.
\end{proof}

\begin{lemma}\label{matrixpart2}
Let $p$ be an odd prime and $d:=p^2-1$. Let $X \to \mathbb P^1$ be the Artin-Schreier cover defined by the equation $y^p-y=-x^{d}-x^{d/2}$. If $\alpha = y^mx^ndx \in \mathcal B_X$ has $m \geq \frac {p-1}{2}$, then $\mathcal C_X(\alpha) \not\in \emph{\text{Span}}(\mathcal C_X(\mathcal B_\alpha))$.
\end{lemma}
\begin{proof}
    The proof of Lemma \ref{matrixpart2} is similar to the proof of Lemma \ref{plusmatrixpart2} with the formula \eqref{exponentsminus} replacing formula \eqref{exponents2}
\end{proof}
    
\begin{proposition}\label{bctheorem}
    Let $p$ be an odd prime. If $d=p^2 - 1$, the Artin-Schreier cover 
    $X\to \mathbb P^1$ defined by $y^p-y=-x^{d}-x^{d/2}$ has a-number equal to $L(d)$.
\end{proposition}
\begin{proof}
    The proof of Proposition \ref{bctheorem} is similar to the proof of Proposition \ref{bc2}, with Lemmas \ref{matrixpart1} and \ref{matrixpart2} playing the role of Lemmas \ref{plusmatrixpart1} and \ref{plusmatrixpart2}.
\end{proof}

\begin{example}
    The Artin-Schreier covers $X \to \mathbb P^1$ defined by $y^{11}-y=-x^{122}-x^{72}$ and $Y\to \mathbb P^1$ defined by $y^{11}-y=-x^{120}-x^{60}$ both have the $a$-number $a(X)=a(Y)=300$, which is equal to the lower bound, $L(120)=L(122)=300$.
\end{example}

\begin{remark}\label{experiment}
For $0\leq n \leq 7$ and for odd prime $p\leq 13$, the Artin-Schreier curves defined by the equation $y^p-y=-x^{np^2-1}-x^{\frac{np^2+(n-1)p-1}{2}}$ have $a$-number equal to the lower bound $L(np^2-1)$. We conjecture that this holds for all $n \in \mathbb N$ and all odd primes $p$. A technique similar to the ones presented in Propositions \ref{bctheorem} and \ref{bc2} might be able to show this, but we present a conceptual approach to generating an infinite family with $a$-number equal to the lower bound in Section \ref{infinite family}.
\end{remark}

\section{An Infinite Family of Curves via Patching}\label{infinite family}
\subsection{Notation}

Notation \ref{first notation} will be used all throughout Section \ref{infinite family}.

\begin{notation} \label{first notation}
Let $p$ be an odd prime $k$ be an algebraically closed field with characteristic $p$. Let $R = k[[t]]$ be the ring of formal power series over $k$ and $K= k((t))$ be the field of fractions of $R$. Set $U=\text{Spec}(k[[u]])$ and  $V=\text{Spec}(k[[v]])$. For a positive integer $e$, set $\Omega_{uv}^e=k[[u,v,t]]/(uv-t^e)$ and let $S^e_{uv} = \text{Spec}(\Omega_{uv}^e)$. 
A relative curve (or $R$-curve) is a flat finitely presented morphism $f:X \to \text{Spec}(R)$ of relative dimension 1. Denote $X_K$ for the generic fiber and $X_k$ for the special fiber of $X$. For a point $u$ on $X$, the \textit{germ} $\hat X_u$ of the curve $X$ at $u$ is the spectrum of the complete local ring of functions of $X$ at $u$ and $\hat{O}_{X,u}$ is the ring of germs of regular functions at $u$.
Let $P_R^e$ be an $R$-curve whose generic fiber is isomorphic to $\mathbb P_k^1$ and whose special fiber consists of projective lines $P_u$ and $P_v$ meeting transversally at a point $b$ where $(u,v)= (\infty, \infty)$ and $\hat {P}_{R,b} \simeq S^e_{uv}$.
\end{notation}

An Artin-Schreier curve can be defined by an equation of the form $y^p-y = f(x)$ where $f(x) \in k(x)$ is nonconstant, and $k(x)$ is the function field of $\mathbb P^1$. We often define a smooth projective curve and its function field by describing its affine parts.  A morphism of curves is a $cover$ if it is finite and generically separable. Let $\phi: Y \to X$ be a $\Zp$-cover of curves branched at the point $u$. Suppose $\eta \in \phi^{-1}(u)$, then the \textit{ramification break} of $\phi$ at $\eta$ is the integer $d=\text{val}(q(\pi_\eta)-\pi_\eta)$, where $q$ is a generator of $\Zp$ and $\pi_\eta$ is a uniformizer of $Y$ at $\eta$.

\subsection{Formal Patching}
Patching is the technique of constructing global covers with certain desired properties by building covers locally and gluing them together. Classically, this was done for varieties over the complex numbers, by building spaces on open sets and gluing them together on the intersections. The main tool behind the classical method of patching was Riemann's Existence Theorem, which showed that this topological construction actually leads to an algebraic variety. Formal patching is a technique that generalizes this approach to work over fields other than the complex numbers using formal geometry. In the case of formal patching, covers are built over formal thickenings and glued on their formal overlap. Grothendieck's Existence Theorem algebraizes this construction. For a brief overview on constructing covers using formal patching see \cite{priesoverview}, and for a more complete introduction to formal patching and its relationship to complex patching see \cite{harbatergalois}.

In this paper, we use formal patching to glue together curves with known ramification breaks and $a$-numbers to form a curve with a larger ramification break and known $a$-number. This section closely follows the ideas presented in \cite{prieswild}. Let $X$ be a connected, reduced, projective $k$-curve and let $\mathbb S$ be a finite closed subset of $\mathbb X$ containing the singular locus of $X$.

\begin{definition}
A \textit{thickening problem} of covers for $(X,\mathbb S)$ consists of the following:
\begin{enumerate}
\item A cover $f:Y\to X$ of geometrically connected reduced projective $k$-curves,
\item For each $s \in \mathbb S$, a Noetharian normal complete local domain $R_s$ with $R \subseteq R_s$ such that $t$ is in the maximal ideal of $R_s$ and a finite generically seperable $R_s$-algebra $A_s$,
\item For each $s \in \mathbb S$, a pair of $k$-algebra isomorphisms $F_s:R_s/(t) \to \hat{O}_{X,s}$ and $E_s:A_s/(t) \to \hat{O}_{Y,s}$ compatible with the inclusion morphisms.
\end{enumerate}
\end{definition}

\begin{definition}
A \textit{thickening} of $X$ is a projective normal $R$-curve $X^*$ such that $X^*_k \simeq X$. We call a thickening problem \textit{$G$-Galois} if $f$ and the inclusion $R_s \subseteq A_s$ are $G$-Galois and $F_s$ is compatible with the $G$-Galois action for all $s\in S$. We call it \textit{relative} if the problem has a thickening $X^*$ of $X$ that is a trivial deformation of $X$ away from $\mathbb S$ such that the pullback of $X^*$ over the complete local ring at a point $s \in \mathbb S$ is isomorphic to $R_s$.
\end{definition}

\begin{definition}
    A \textit{solution} to a thickening problem of covers is a cover $f^*:Y^* \to X^*$ of projective normal $R$-curves, where $X^*$ is a thickening of $X$, whose closed fiber is isomorphic to $f$, whose pullback to the formal completion of $X^*$ along $X'=X-\mathbb S$ is a trivial deformation of the restriction of $f$ over $X'$, and whose pullback over the complete local ring at a point $s\in \mathbb S$ is isomorphic to $R_s \subseteq A_s$ with all isomorphisms being compatible.
\end{definition}

\begin{theorem}\label{harbaterstevenson}
    Every $G$-Galois thickening problem for covers has a $G$-Galois solution. If the thickening problem is relative then the solution is unique.
\end{theorem}

\begin{proof}
    \cite[Theorem 4]{harbaterstevenson}.
\end{proof}

\subsection{Patching curves with a specific congruence}\label{patching curves}

\begin{proposition} \label{crossing}
    Following Notation \ref{first notation}, let $\phi_1: X \to U$, $\phi_2: Y\to V$ be $\mathbb Z/p\mathbb Z$-Galois covers of normal connected germs of curves with ramification breaks $j_1$ and $j_2$ such that $j_1+j_2 \equiv 0 \pmod p$. Let $e=j_1+j_2$. Then there exists $\mathbb Z/p\mathbb Z$-Galois cover $\phi_R: W_R \to S^e_{uv}$ of irreducible germs of $R$-curves with the properties:
    \begin{enumerate}
    \item The cover $\phi_R$ is branched at one $R$-point, $b_R$.
    \item After normalization, the pullbacks of the special fiber of $\phi_R$ to U and V are isomorphic to $\phi_1$ and $\phi_2$ away from $b_R$.
    \item The generic fiber $\phi_K: W_K \to S^e_{uv,K}$ of $\phi_R$ is a $\mathbb Z/p\mathbb Z$-Galois cover of normal irreducible germs of curves whose branch locus is $b_K:=b_R \times_R K$.
    \item The cover $\phi_K$ has ramification break $e-1$ over the branch point.
    \end{enumerate}
\end{proposition}

\begin{proof}
    We adapt the proof from \cite[Proposition 2.3.4]{prieswild}.
    There exists automorphisms $A_u$ and $A_v$ of fixing the closed points of $U$ and $V$ such that $A_u^*\phi_1$ and $A_v^*\phi_2$ can be given by the equations $x^p-x = u^{-j_1}$ and $y^p - y = v^{-j_2}$ respectively (see \cite[$\S 10.4$]{artin}). Denote the transformed covers as $\phi_1'$ and $\phi_2'$ respectively.
    We may suppose the Galois action of $\phi_1'$ maps $x \mapsto x+1$ and the Galois action of $\phi_2'$ maps $y \mapsto y +a$ for some $a \in \mathbb F_p^\times$.\\
    Let $\phi_R: W_R \to S^e_{uv}$ be the cover defined by the equation 
    \begin{equation}\label{phi_R equation}
    z^p-z=(u^{j_1}+av^{j_2}+d_0t)^{-1}
    \end{equation}
    for some $d_0 \in \Omega_{uv}^e$. Note that after reducing $\pmod {(v,t)}$, the equation becomes $z^p - z =u^{j_1}$, which is identical to the equations defining $\phi_1'$. So the normalization of the reduction is isomorphic to $\phi_1'$. Likewise the normalization of the reduction $\pmod {(u,t)}$ is isomorphic to $\phi_2'$. The Galois action of $\phi_R$ sends $z \mapsto z+1$, which reduces to the Galois action on $\phi_1'$ and $\phi_2'$ respectively. Hence the pullbacks of the normalization of the special fiber of $\phi_R$ to $U$ and $V$ are isomorphic to $\phi_1'$ and $\phi_2'$ as covers away from the branch point.
    
    \noindent Set $$d_0=\frac{(u+a_1t^{j_2})^{e}-u^{e}-at^{j_2e}}{u^{j_2}t}$$\\
    with $a_1=\sqrt[e]{a^{-1}}$. Note that $d_0 \in \Omega_{uv}^e$.\\
    Plugging $d_0$ into \eqref{phi_R equation} and simplifying gives 
    \begin{equation}\label{plugged in d0}
    z^p-z=u^{j_2}(u+a_1t^{j_2})^{-e}.
    \end{equation}
    Note that $u$ has no zero or pole in $\Omega_{uv,K}^e$. Thus $\phi_K$ has a unique branch point given by $u=-a_1t^{j_2}$ and $v=-t^{j_1}/a_1$. This $K$-point specializes to the branch point $(u,v)=(0,0)$ of $\phi_k$. Thus $\phi_R$ is branched at only one $R$-point with this $d_0$, given by $(u,v)=(-a_1t^{j_2},-a_1^{-1}t^{j_1})$. We call this $R$-point $b_R$.\\
    Note that \eqref{plugged in d0} is irreducible in $R$ since the right hand side is not of the form $\alpha^p - \alpha$. Hence $W_R$ and $W_K$ are irreducible.
    From equation \eqref{plugged in d0}, $u^{j_2}(u+a_1t^{j_2})^{-e}$ is an $e$th power of a uniformizer. Hence the largest that the ramification break can be is $e$. By Proposition \ref{fact}, the lower ramification break must be greater than or equal to $e-1$. Since $p$ cannot divide the ramification break, the lower ramification break on the generic fiber must be $e-1$.
\end{proof}

\begin{proposition}\label{fact}
    Given covers $\phi_1$ and $\phi_2$ with ramification breaks $j_1$ and $j_2$, the ramification break on the generic fiber of a flat deformation of these two covers, $\phi$, will have ramification break $j \geq j_1 + j_2 - 1$.
\end{proposition}

\begin{proof}
    For an Artin-Schreier curve, the genus and ramification break are related by the formula $g=\frac{p-1}{2}(j-1)$, a corollary of the Riemann-Roch theorem. The genus of the special fiber of the deformation is $g_\phi=\frac{p-1}{2}((j_1-1)+(j_2-1))+\alpha$, where $\alpha$ is some nonnegative integer, contributed by the singularity. Note that the genus is constant in a flat family of curves (see Corollary 9.9.10 in \cite{hartshorne}). From the above equality we then get the inequality $j_\phi -1 \geq (j_1-1)+(j_2-1)$, which gives the fact.
\end{proof}

\begin{proposition} \label{realize}
    Let $\phi_1: X_1 \to \mathbb P^1$, $\phi_2: X_2 \to \mathbb P^1$ be $\mathbb Z/p\mathbb Z$-Galois covers branched at only one point with ramification breaks $j_1$ and $j_2$ such that $j_1+j_2 \equiv 0 \pmod p$. Then there exists a $\mathbb Z/p\mathbb Z$-Galois cover of curves $\phi: Y \to \mathbb P_k^1$ with the following properties:
    \begin{enumerate}
    \item $\phi$ has exactly one branch point.
    \item There is a single ramification point of $\phi$ above the branch point and its ramification break is $e = j_1+j_2-1$.
    \item $Y$ is smooth and connected.
    \item $\phi$ is a fiber of a cover $\phi_R:Y_R \to P_R^e$ with the property that the pullbacks of the special fiber of $\phi_R$ to $P_u$ and $P_v$ are isomorphic to $\phi_1$ and $\phi_2$ away from the branch point after normalization.
    \end{enumerate}
\end{proposition}

\begin{proof}
    We adapt the proof of Theorem 2.3.7 from \cite{prieswild}. From Notation \ref{first notation}, label $X^*=P^e_R$,  $\mathbb S = \{b\}$, and $G =\Zp$. 
    There exists ramified points $\eta_1\in\phi_1^{-1}(u)$ and $\eta_2 \in \phi_2^{-1}(v)$. Consider the $\Zp$-Galois covers of germs of curves $\hat{\phi}_1: \hat{X}_{\eta_1} \to U$ and $\hat{\phi}_2:\hat{Y}_{\eta_2} \to V$. 
    Apply Proposition \ref{crossing} to $\hat{\phi}_1$ and $\hat{\phi}_2$ to get $\Zp$-Galois cover $\hat{\phi}_R:W_R \to S^e_{uv}$ which is ramified at one point with ramification break $e$. 
    Note that $\hat{\phi}_R$ corresponds to an inclusion of rings.
    Consider the cover $\phi_k$ of the special fiber of $P_R^e$ which restricts to $\phi_1$ over $P_u$ and to $\phi_2$ over $P_v$. 
    Form a relative $G$-Galois thickening problem using $\phi_k$ and $\hat{\phi}_R$ and the isomorphisms from the pullbacks of the special fiber of $\hat{\phi}_R$ to $\phi_1$ and $\phi_2$. 
    By Theorem \ref{harbaterstevenson}, this problem has a solution, a $\Zp$ cover $\phi_R: Y_R \to P_R^{e}$. 
    The closed fiber of $\phi_R$ is isomorphic to $\hat{\phi}_R$ and $\phi_R$ is isomorphic to the trivial deformation away from the closed point. 
    Hence, $Y_K$ is smooth because $W_K$ and the trivial deformation are smooth.
    
    Choose a subring $O \subseteq R$ finitely generated over $k$ with $O\neq k$, such that $\phi_R$ can be defined over $\text{Spec}(O)$. Note that such a subring exists because $\phi_R$ is defined using only finitely many elements of $R$.
    Since $k$ is algebraically closed, there are infinitely many $k$-points of $\text{Spec}(O)$. 
    Let $L$ be the set of $k$-points, $x$, of $\spec{O}$ such that $\phi_x$ is not a $\Zp$ cover of smooth connected curves. 
    Note that the closure $\overline L \neq \spec{O}$ since $Y_K$ is smooth and irreducible \cite[\href{https://stacks.math.columbia.edu/tag/055G}{Lemma 055G}]{stacks-project}.
    Let $\alpha \in \spec{O}\setminus L$ be a $k$-point and $\phi := \phi_\alpha: Y \to X_k$ be the fiber over $\alpha$. 
    The map $\phi$ inherits properties 1 and 2 from $\phi_R$ and is smooth and connected by construction.
\end{proof}

\begin{proposition}\label{additive a-number}
In the notation of Proposition~\ref{realize}, the special fiber of $\phi_R$ is a cover of stable curves where the $a$-number of the cover is $a(X_1) + a(X_2)$.
\end{proposition}

\begin{proof}
Away from the reduction of the branch point, in the special fiber the cover $Y_k$ is the disjoint union of $X_1$ and $X_2$ with the ramification points removed.  The genus of $Y_k$ is the sum of the genera of $X_1$ and $X_2$ with a contribution from the singularity caused by gluing $X_1$ and $X_2$.  This is the same setup as Proposition~\ref{fact}.  However, we know that 
\[
g(X_1) + g(X_2) = \frac{p-1}{2}((j_1-1)+(j_2-1)) = \frac{p-1}{2}  (j_1+j_2-1) = g(Y_R)
\]
so the singularity makes no contribution.  Thus the singularity is an ordinary double point, so the generalized Jacobian is the product of the Jacobians of $X_1$ and $X_2$ and has no toric part.  In particular, the $a$-number is $a(X_1) + a(X_2)$.
\end{proof}

\begin{proposition}\label{a-number godown}
Let $S$ be irreducible with generic point $\eta$, and $\pi : X \to S$ be a smooth family of projective curves over $S$.  Then for any point $s \in S$, $a(X_\eta) \leq a(X_s)$.
\end{proposition}

\begin{proof}
The Cartier operator can be viewed as a map of $\mathcal{O}_{X/S}$-modules $F_* \Omega^1_{X/S} \to \Omega^1_{X/S}$, whose kernel is a coherent sheaf $\mathscr{F}$ such that for $s \in S$
\[
a(X_s) = \dim_{k(s)} H^0(X_s,\mathscr{F}_s).
\]
Note that $\mathscr{F}$ is flat over $X$ and hence over $S$ as it is a kernel of a map of locally free sheaves.
By the semicontinuity theorem \cite[Theorem III.12.8]{hartshorne}, the $a$-number is an upper semi-continuous function on $S$, i.e. $a(X_\eta) \leq a(X_s)$.
\end{proof}

\begin{proposition}\label{generalized family}
    For a prime $p$, if there exists a $\Zp$-Galois cover $X \to \mathbb P^1$ with ramification break $d\equiv -1 \pmod p$ with $a$-number equal to $L(d)$, then for any $e\geq d$ with $e \equiv d \pmod{p^2}$ there exists a $\Zp$-Galois cover $X \to \mathbb P^1$ branched at one point with ramification break $e$ and $a$-number equal to $L(e)$.
\end{proposition}

\begin{proof}
    We prove Proposition \ref{generalized family} by induction. By assumption there exists a $\Zp$-Galois cover $X \to \mathbb P^1$ with ramification break $d\equiv -1 \pmod p$ with $a$-number equal to $L(d)$. Assume $X_1\to \mathbb P^1$ is a $\Zp$-Galois cover of curves with ramification break $e \geq d$ with $e\equiv d \pmod {p^2}$ and $a$-number $a(X_1) = L(e)$. Let $X_2 \to \mathbb P^1$ be the Artin-Schreier cover defined by $y^p-y=-x^{p^2+1}-x^{\frac{p^2+1}{2}+p}$. By Proposition \ref{bc2}, $X_2$ has $a$-number $a(X_2)=L(p^2+1)$. Using Proposition \ref{realize}, there exists a $\Zp$-Galois cover $X \to \mathbb P^1$ with ramification break $e+p^2$ and, by Proposition \ref{additive a-number} and \ref{a-number godown}, $a$-number $a(X)=a(X_1)+a(X_2)$. By Lemma \ref{addingpsquared}, $L(e+p^2)=L(e)+L(p^2+1)=a(X_1)+a(X_2)$. Hence the $a$-number of $X$ is equal to the lowerbound, $a(X)=L(e+p^2)$. Hence, by induction, for any integer $e \geq d$ with $e \equiv d \pmod{p^2}$ there exists a $\Zp$-Galois cover $X \to \mathbb P^1$ with ramification break $e$ and $a$-number equal to $L(e)$.
\end{proof}

\begin{theorem}
\label{family}
    For any odd prime $p$ and any positive $d \equiv -1 \pmod {p^2}$ there exists a $\Zp$-Galois cover $X \to \mathbb P^1$ branched at one point with ramification break $d$ and $a$-number equal to $L(d)$.
\end{theorem}

\begin{proof}
    By Proposition \ref{bctheorem}, the cover defined by $y^p-y=-x^{p^2-1}-x^{\frac{p^2-1}{2}}$ has $a$-number equal to $L(p^2-1)$. By Proposition \ref{generalized family}, for any $d>p^2-1$ with $d\equiv -1 \pmod {p^2}$, there exists $\Zp$-Galois cover $X \to \mathbb P^1$ with ramification break $d$ and $a$-number equal to $L(d)$.
\end{proof}

\begin{theorem}\label{farnell family}
       For any odd prime $p$ and any positive $d \equiv p-1 \pmod {p^2}$ there exists a $\Zp$-Galois cover $X \to \mathbb P^1$ branched at one point with ramification break $d$ and $a$-number equal to $L(d)$. 
\end{theorem}

\begin{proof}
    Choose $f\in k[x]$ such that $\deg(f)=p-1$. By Proposition \ref{farnellresult}, the cover defined by $y^p-y=f$ has $a$-number equal to $L(p-1)$. By Proposition \ref{generalized family}, for any $d>p-1$ with $d\equiv p-1 \pmod {p^2}$, there exists $\Zp$-Galois cover $X \to \mathbb P^1$ with ramification break $d$ and $a$-number equal to $L(d)$.
\end{proof}

\begin{theorem}\label{code examples}
    For any odd prime $p \leq 23$ and any positive $d \equiv -1 \pmod p$, there exists a $\Zp$-Galois cover $X \to \mathbb P^1$ branched at one point with ramification break $d$ and $a$-number equal to $L(d)$.
\end{theorem}

\begin{proof}
    Let $d$ be a positive integer with $d \equiv -1 \pmod p$. Let $1 \leq k \leq p$ such that $kp-1 \equiv d \pmod {p^2}$. By explicit computation, we found a polynomial $f$ with degree $kp-1$ such that the cover defined by $y^p-y=f(x)$ has $a$-number equal to $L(kp-1)$. A list of the polynomials and the MAGMA code can be found at \cite{github} or attached to the arXiv submission. By Proposition \ref{generalized family}, since $d>kp-1$ and $d\equiv kp-1 \pmod {p^2}$, there exists $\Zp$-Galois cover $X \to \mathbb P^1$ with ramification break $d$ and $a$-number equal to $L(d)$.
\end{proof}

\bibliographystyle{halpha-abbrv}
\bibliography{citations}

\end{document}